 \documentclass[letterpaper, 10 pt, conference]{ieeeconf}  % Comment this line out if you need a4paper

\IEEEoverridecommandlockouts                              % This command is only needed if 
\usepackage{graphicx}
\usepackage{subfig}
\usepackage{svg}

\usepackage{todonotes}
\usepackage{bbm}
\usepackage{bm}

\usepackage{pgf}
\usepackage{tikz}
\usetikzlibrary{arrows,automata, positioning}
\usepackage[latin1]{inputenc}
\usepackage{verbatim}

\usepackage{mathtools}
\usepackage{stmaryrd}

\usepackage[]{algorithm2e}

\usepackage{amssymb,latexsym,amsfonts,amsmath,amscd, amsthm} 
\usepackage{acronym}
\usepackage{comment}
\usepackage{lineno}
\newif\ifuseboldmathops
\newif\ifuseittextabbrevs
\useboldmathopstrue   % comment out to use mathbb
\ifuseittextabbrevs

	\newcommand{\ie}{{\it i.e.}}

\else

	\newcommand{\ie}{i.e.}

\fi

\ifuseboldmathops
	\newcommand{\reals}{\mathbf{R}}

\else
	\newcommand{\reals}{\mathbb{R}}

\fi

\ifuseboldmathops

\else

\fi

\ifuseboldmathops
	\newcommand{\Expect}{\mathop{\bf E{}}\nolimits}
	
\else
	\newcommand{\Expect}{\mathop{\mathbb{E}{}}\nolimits}
	
\fi

\ifuseboldmathops

\else

\fi

\newcommand{\Eventually}{\Diamond \, }

\newcommand{\until}{\mbox{$\, {\sf U}\,$}}

\newcommand{\sink}{\mathsf{sink} }

\newcommand{\BoolTrue}{\mbox{\sf true}}

\acrodef{mcmc}[mcmc]{Monte Carlo Markov chain}

\newcommand{\norm}[1]{\lVert#1\rVert}

\newcommand{\abs}[1]{\lvert#1\rvert}

\newcommand{\dist}[1]{\mbox{Dist}(#1)}
\newcommand{\prob}{\mbox{Prob}}

\newcommand{\calO}{\mathcal{O}}

\newcommand{\calL}{\mathcal{L}}

\newcommand{\calA}{\mathcal{A}}

\newcommand{\calAP}{\mathcal{AP}}

\newcommand{\init}{\mathsf{Init}}

\newcommand{\prog}{\mathsf{Prog}}

\newcommand{\inside}[1]{\llbracket
{#1} \rrbracket
}

\newcommand{\nat}{\mathbb{N}}

\newcommand{\indicator}{\mathbf{1}}

\acrodef{mdp}[MDP]{Markov decision process}
\acrodef{dfa}[DFA]{deterministic finite-state automaton}
\acrodef{ltl}[LTL]{linear temporal logic}
\acrodef{ag}[AG]{Assume-Guarantee}
\acrodef{ssp}[SSP]{Stochastic Shortest Path}
\acrodef{gcd}[GCD]{Generalized Conjunction-Disjunction}

 \newtheorem{definition}{Definition}
 \newtheorem{example}{Example}
\newtheorem{problem}{Problem}
\newtheorem{lemma}{Lemma}

\newcommand{\unsafe}{\mathsf{Unsafe}}

\newcommand{\goal}{\mathsf{Goal}}

\acrodef{ltl}[LTL]{linear temporal logic formula}
\acrodef{mdp}[MDP]{Markov decision process}
\acrodef{smdp}[Semi-MDP]{Semi-Markov decision process}

\acrodef{scltl}[sc-LTL]{}

\newtheorem{remark}{Remark}
\newtheorem{assumption}{Assumption}

\title{\LARGE \bf
Compositional planning in Markov decision processes: Temporal abstraction meets  generalized logic composition
}

\author{Xuan Liu and Jie Fu % <-this % stops a space
\thanks{X. Liu and J. Fu are with the Department of Electrical and Computer Engineering, Robotics Engineering Program, Worcester Polytechnic Institute, Worcester, MA. USA.  {\tt\small xliu9, jfu2@wpi.edu}
}
\thanks{This material is based upon work supported by the National Science Foundation under Grant No. \#1728412.}}

\begin{document}

\maketitle
\thispagestyle{empty}
\pagestyle{empty}

%%%%%%%%%%%%%%%%%%%%%%%%%%%%%%%%%%%%%%%%%%%%%%%%%%%%%%%%%%%%%%%%%%%%%%%%%%%%%%%%
\begin{abstract}
In hierarchical planning for Markov decision processes (MDPs), temporal abstraction allows planning with macro-actions that take place at different time scale in form of sequential composition. In this paper, we propose a novel approach to compositional reasoning and hierarchical planning for MDPs under co-safe temporal logic constraints. In addition to sequential composition, we introduce a composition of policies based on generalized logic composition: Given sub-policies for sub-tasks and a new task expressed as logic compositions of subtasks, a semi-optimal policy, which is optimal in  planning with only sub-policies, can be obtained by simply composing sub-polices. 
Thus, a synthesis algorithm is developed to compute optimal policies efficiently by planning with primitive actions, policies for sub-tasks, and the compositions of sub-policies, for maximizing the probability of satisfying constraints specified in the fragment of co-safe temporal logic. We demonstrate the correctness and efficiency of the proposed method in  stochastic planning examples with a single agent and multiple task specifications.
\end{abstract}

%%%%%%%%%%%%%%%%%%%%%%%%%%%%%%%%%%%%%%%%%%%%%%%%%%%%%%%%%%%%%%%%%%%%%%%%%%%%%%%%
\section{INTRODUCTION}

 Temporal logic is an expressive language to describe  desired system properties: safety, reachability, obligation, stability, and liveness \cite{manna2012temporal}. The algorithms for planning and probabilistic verification with temporal logic constraints have developed, with both centralized  \cite{baier2008principles,ding2011mdp, lahijanian2012temporal} and distributed methods \cite{fu2015optimal_distributed}. Yet, there are two main barriers to practical applications: 1) The issue of scalability: In temporal logic constrained control problems,  it is often necessary to introduce additional memory states for keeping track of the evolution of state variables with respect to these temporal logic constraints.  The additional memory states grow exponentially (or double exponentially depending on the class of temporal logic) in the length of a specification \cite{gastin2001fast} and make synthesis computational extensive. 2) The lack of flexibility: With a small change in the specification, a new policy may need to be synthesized from scratch.

To improve scalability for planning given complex tasks, composition is an idea exploited in temporal abstraction and hierarchical planning in \ac{mdp} \cite{Silver2012composition, bacon2017option}. To accomplish complex tasks, temporal abstraction allows planning with macro-actions---policies for simple subtasks---with different time scales. 
 A well-known hierarchical planner is called the  options framework \cite{parr1998reinforcement,sutton1999between,precup2000temporal}. An option is a pre-learned policy for a subtask given the original task that can be completed by temporally abstracting subgoals and sequencing the subtasks' policies.

Once an agent learns the set of options from an underlying MDP, it can use conventional reinforcement learning to learn the global optimal policy with the original action set augmented with the set of options, also known as sub-policies or macro-actions. 
In light of the options framework, hierarchical planning in MDPs is evolving rapidly, with both model-free \cite{kulkarni2016hierarchical} and model-based \cite{Silver2012composition, Sorg}, and with many practical applications in robotic systems \cite{konidaris2011autonomous, konidaris2012robot}. The option-critic method \cite{bacon2017option} integrates approximate dynamic programming \cite{bertsekas2008neuro} with the options framework to further improve its scalability. 

Since temporal logic specifications describe temporally extended goals and the options framework uses temporally abstracted actions, it seems that applying the options framework to planning under temporal logic constraints is straightforward. However, a direct application does not take full advantages of various compositions observed in temporal logic. The options framework captures the \emph{sequential composition}. However, it does not consider \emph{composition for conjunction or disjunction in logic}. In this paper, we are interested in answering two questions: Given two options that maximize the probabilities of $\varphi_1$ and $\varphi_2$, is there a way to compose these two options to obtain a ``good enough'' policy for maximizing the probability of $\varphi_1\land\varphi_2$, or $\varphi_1\lor \varphi_2$? If there exists a way to compose, what shall be the least set of options that one needs to generate? Having multiple ways of composition enables more flexible and modular planning given temporal logic constraints. For example, consider a specification $\Eventually ((R_1\lor R_2) \land \Eventually R_3)$, \ie, eventually reaching the region $R_3$ after visiting any of the two regions $R_1$ and $R_2$. With composition for sequential tasks only, we may generate an option that maximizes the probability of reaching $R_1\lor R_2$ and an  option that maximizes the probability of reaching $R_3$. With both compositions of sequencing, conjunction, and disjunction of tasks, we may generate options that maximize the probabilities of reaching $R_1$, $R_2$, and $R_3$, respectively, and compose the first two to obtain the option for $ \Eventually R_1\lor \Eventually  R_2$. In addition, we can compose options to not only for $\Eventually R_1\lor \Eventually  R_2$ but also have 
$\Eventually R_1\lor \Eventually  R_3$, 
$\Eventually R_1\lor \Eventually  R_2 \lor \Eventually R_3$, etc. When the task changes to  $\Eventually ((R_1\lor \bm{R_3}) \land \Eventually \bm{R_2})$, the new option for  $ \Eventually R_1\lor \Eventually  R_3$ needs not to be learned or computed, but composed.

In a pursuit to answering these two questions, the contribution of this paper is two-fold: We develop an automatic decomposition procedure to generate a \emph{small set} of primitive options from a given co-safe temporal logic specification. We formally establish an equivalence relation between  Generalized Conjunction/Disjunction (GCD) functions \cite{Dujmovi2007} in quantitative logic and composable solutions of \ac{mdp} using entropy-regulated Bellman operators \cite{Silver2012composition,nachum2017bridging}. This equivalence enables us to compose policies for simple formulas/tasks to maximize the probability for satisfying formulas obtained via GCD composition of these simple formulas. Last, we use these novel composition operations to develop a hierarchical planning method for \ac{mdp} under co-safe temporal logic constraints. We demonstrate the efficiency and correctness of the proposed method with several examples.

\section{Preliminaries} %and notations

Notation:
Let $\nat$ be the set of nonnegative integers.
Let $\Sigma$ be an alphabet (a finite set of symbols). Given $k\in \nat$, $\Sigma^k$ indicates a set of strings with length $k$, $\Sigma^{\leq k}$ indicates a set of finite strings with length smaller or equal to $k$, and $\Sigma^0 = \lambda$ is the empty string. $\Sigma^\ast$ is the set of all finite strings (also known as Kleene closure of $\Sigma$).  Given a set $X$, let $\dist{X}$ be a set of probabilistic distributions with $X$ as the support.

In this paper, we consider temporal logic formulas for specifying desired properties in a stochastic system. Given a set $\calAP$ of atomic propositions, a syntactically co-safe linear temporal logic (sc-LTL) \cite{kupferman2001model} formula over $\calAP$ is inductively defined as follows:
\begin{center}
    $\varphi \coloneqq \BoolTrue | p | \neg p |\varphi_1\land\varphi_2 | \varphi_1\lor\varphi_2 | \bigcirc\varphi | \varphi_1\until \varphi_2 $.
\end{center}

The above formula is composed of unconditional true $\BoolTrue$, state predicates $p$ and its negation $\lnot p$, conjunction ($\land$) and disjunction ($\lor$), temporal operators ``Next'' ($\bigcirc$), and ``Until'' ($\until$). Temporal operator ``Eventually'' ($\Eventually$) is defined by: $\Eventually \phi \coloneqq \BoolTrue \until \phi$. However, temporal operator ``Always'' cannot be expressed in sc-LTL.
A detailed description of the
syntax and semantics of sc-LTL can be found in \cite{pnueli1977temporal}. An sc-LTL formula $\varphi$ is evaluated over finite words. In addition to the above notation, we use a backslash ($\backslash$) between two propositions to represent the logic exclusion, \ie, rewriting $\varphi_1\land\lnot\varphi_2$ to $\varphi_1\backslash\varphi_2$.

Given an sc-LTL formula $\varphi$, there exists a \ac{dfa} that accepts all strings that satisfy the formula $\varphi$ \cite{gastin2001fast}. The \ac{dfa} is a tuple
$\calA_\varphi=\langle Q, \Sigma, \delta, q_0, F \rangle$, where $Q$
is a finite set of states, $\Sigma = 2^\calAP$ is a finite
alphabet, $\delta: Q\times \Sigma \rightarrow Q$ is a
\emph{deterministic} transition function such that when the symbol
$\sigma \in \Sigma$ is read at state $q$, the automaton makes a
deterministic transition to state $\delta(q,\sigma)=q'$, $q_0 \in Q$
is the initial state, and $F\subseteq Q$ is a set of final,
\emph{accepting} states.  The transition function is extended to a
sequence of symbols, or a \emph{word}
$w=\sigma_0\sigma_1\ldots\in\Sigma^\ast$, in the usual way:
$\delta (q,\sigma_0 v) = \delta(\delta(q,\sigma_0),v)$ for
$\sigma_0 \in \Sigma$ and $v\in \Sigma^\ast$. A finite
word $w$ satisfies $\varphi$ if and only if $\delta(q_0, w)\in F$. The
set of words satisfying $\varphi$ is the \emph{language} of the
automaton $\calA_\varphi$, denoted $\calL(\calA_\varphi)$.

We consider stochastic systems modeled by \ac{mdp}. The specification is given by an sc-LTL formula and related to  paths in an MDP via a labeling function. 

 \begin{definition}[Labeled \ac{mdp}]   
\label{def:labeledmdp}
  A labeled MDP is a tuple $ M= \langle S, A,  \mu_0, P,  \calAP,
  L \rangle $ where $S$ and $A$ are finite state and action sets.
  $\mu_0\in \dist{S}$ is the initial state distribution.  The transition probability
  function $P: S \times A \times S \rightarrow [0,1]$ is defined
  such that $\sum_{s' \in S} P(s,a,s') \in \{0,1\}$ for any state
  $s\in S$ and any action $a \in A$. %$\gamma \in (0,1]$ is a discounting factor, % and $r:S\times A\rightarrow \reals$ is the reward function, 
  $\calAP$ is a finite
  set of atomic propositions  
  and $L: S \rightarrow 2^{\calAP}$ is a 
  labeling function which assigns to each state $s \in S$ a set of
  atomic propositions $L(s)\subseteq \calAP$ that are valid at the
  state $s$. $L$ can be extended to state sequences in the usual way,
  i.e., $L(\rho_1\rho_2)=L(\rho_1)L(\rho_2)$ for $\rho_1,\rho_2\in
  S^\ast$.
\end{definition}

A finite-memory, stochastic policy in the MDP is a function $\pi:S^\ast \rightarrow \dist{A}$. A Markovian, stochastic policy in the MDP is a function $\pi:S\rightarrow \dist{A}$. Given an MDP $M$ and a policy $\pi$, the policy induces a Markov chain $M^\pi = \{s_t|t = 1, \ldots, \infty\}$ where  $s_k$ as the random variable for the $k$-th state in the Markov chain $M^\pi$ and it holds that $s_0 \sim \mu_0$ and $s_{i+1} \sim P(\cdot |s_i,a_i)$ and $a_i\sim\pi(\cdot | s_i)$.

Given a finite (resp. infinite) path $\rho  =s_0s_1\ldots s_N \in S^\ast$ (resp. $\rho \in S^\omega$), we obtain a sequence of labels $L(\rho) = L(s_0)L(s_1)\ldots L(s_N)\in \Sigma^\ast$ (resp. $L(\rho) \in \Sigma^\omega$). A path $\rho$ satisfies the formula $\varphi$, denoted $\rho\models \varphi$, if and only if $L(\rho) \in \calL(\calA_\varphi)$. Given a Markov chain induced by policy $\pi$,  the probability of satisfying the specification, denoted $\prob(M^\pi\models \varphi)$ is the sum of the probabilities of paths satisfying the specification.
\[
\prob(M^\pi \models \varphi) \coloneqq \Expect\left[\sum_{t=0}^\infty \indicator(\bm \rho_t\models \varphi)\right].
\]
where ${\bm \rho}_t =s_0s_1\ldots s_t$ is a path of length $t$ in $M^\pi$.

We relate each subset  $\sigma  \in 2^\calAP$ of atomic propositions to a propositional logic formula $ \land_{p\in \sigma} p \land (\lor_{p'\in \calAP\setminus \sigma} \neg p')$. A set of states satisfying the propositional logic formula for $\sigma$ is denoted $\inside{\sigma}$.   Given a subset of proposition set $C\subseteq 2^\calAP$, let $\inside{C} = \bigcup_{c\in C} \inside{c}$.  Slightly abusing the notation, we use $\sigma$ to refer to  the propositional logic formula that $\sigma$ corresponds to. 
The optimal planning problem for \ac{mdp} under sc-LTL constraints is defined as follows.
 
\begin{problem}
\label{problem1}
Given an MDP and an sc-LTL formula $\varphi$, design a policy $\pi$ that maximizes the probability of satisfying the specification, \ie, 
$
\pi\leftarrow \arg\max_\pi \prob(M_{\pi} \models \varphi).
$
\end{problem}
Problem~\ref{problem1} can be solved with dynamic programming methods in a product MDP. The idea is to augment the state space of the MDP with additional memory states---the states in the automaton $\calA_\varphi$, and reformulate the problem into a stochastic shortest path problem in the product MDP with the augmented state space. The reader is referred to \cite{ding2011mdp} for more details. In this paper, our goal is to develop an efficient and hierarchical planner for solving Problem~\ref{problem1}.

  \begin{remark}
  The extension from sc-LTL to the class of LTL formulas can be made by expressing the specification formula using a deterministic Rabin automaton \cite{duret2016spot, giannakopoulou2001automata} and perform two-step synthesis approach: The first step is to compute the maximal accepting end components, and the second step is to solve the \ac{ssp} MDP in the product MDP (assigning reward $1$ to reaching a state in any maximal accepting end component). The details of the method can be found in \cite{Luca1995,ding2011mdp, ding2011ltl}. Particularly, the tools facilitate symbolic computation of maximal accepting end components have been developed \cite{Chatterjee2013}. 
  In the scope of this paper, we only consider sc-LTL formulas. Yet, the generalization can be made to handle planning for general LTL formulas with a similar two-step approaches. 
  \end{remark}
  
\section{Hiererachical and compositional planning under sc-LTL constraints}
  In this section, we present a compositional planning method for solving Problem~\ref{problem1}.  First, we propose a task decomposition method to identify a set of modular and reusable \emph{primitive options}. Second, we establish a relation between logical conjunction/disjunction and composition of primitive options. Building on the options framework, we develop a hierarchical and compositional planning method for temporal logic constrained stochastic systems.

\subsection{Automata-guided generation of primitive options}
\label{sec:primitive_options}

We present a procedure to decompose the task $\varphi$ in sc-LTL into a set of primitive tasks. These primitive tasks will be composed in Sec~\ref{subsec: composition} to generate the set of options in hierarchical planning. 

Given a specification automaton $\calA_\varphi = (Q, \Sigma, \delta, q_0, F)$, let the rank of a state be the minimal number of transitions to the set $F$ of final states. Let $L_k$ be the set of states of rank $k$. Thus, we have \begin{itemize}
    \item $L_0=F$, and
    \item $L_k = \{q\mid \exists w\in \Sigma^k, \delta(q, w)\in F \text{ and }\forall \ell <k,  \forall w \in  \Sigma^\ell,\delta(q,w)\notin F \}$.
\end{itemize}
By definition, if \ac{dfa} $\calA_\varphi$ is coaccessible, \ie, for every state $q\in Q$ there is a word $w$ that takes us from $q$ to a final state, then  for any state $q\in Q$, there exists $L_k$ with a finite rank $k$ that includes  state $q$. Any \ac{dfa} can be made coaccessible by trimming \cite{sakarovitch2009elements}. Finally, for a coaccessible \ac{dfa}, we introduce a sink state to make it complete: For a state $q$ and symbol $\sigma \in \Sigma$, if $\delta(q,\sigma)$ is undefined, then let $\delta(q,\sigma)=\sink$.

 Based on the ranking, for each state $q\in Q$, we distinguish two types of transitions from the state:
\begin{itemize}
    \item A transition is  \emph{progressing}: $q\xrightarrow{\sigma} q'$ and if $q\in L_k$ then $q'\in L_{k-1}$. 
    \item A transition is \emph{unsafe}: $q\xrightarrow{\sigma}\sink$, where $\sink$ is a non-accepting state with self-loops on all symbols.
\end{itemize}
Note that the \ac{dfa}  may have self-loops which are not included in either progressing transition or unsafe transitions. However, we shall see later that ignoring these self-loops will not affect the optimality of the planning algorithm.

A state may have multiple progressing and unsafe transitions. Let  ${\unsafe}(q)$ be the set of labels for unsafe transitions on $q$. Let $\prog(q) $ be the set of labels for progressing transitions on $q$. 
  A conditional reachability formula is defined for $q$ as:
\[ 
\neg \varphi_{\unsafe(q)} \until \varphi_{\prog(q)}, \] 
where $\varphi_{\unsafe(q)} = \land_{\sigma \in {\unsafe}(q)} \sigma$ and $\varphi_{\prog} = \lor_{\sigma \in {\prog}(q)} \sigma $ and $\sigma \in \Sigma$.  This subformula is further decomposed into: 
\[ 
\varphi^q_i \coloneqq \neg \varphi_{\unsafe(q)} \until \sigma_i, \quad \text{ for each } \sigma_i \in \prog(q).  \] 
We define the decomposition of $\varphi$ as the collection of conditional reachability formulas 
\[
\Phi^{cr}= \{\varphi^q_i, q\in Q \mid i=1,\ldots,\abs{ \prog(q)}\} \text{.}
\] 

Next, we prune $\Phi^{cr}$ to obtain the set $\Phi\subseteq \Phi^{cr}$  of \emph{primitive tasks}: $\phi \in \Phi\cap \Phi^{cr}$ if and only if there does not exist a set of formulas $\phi_i \coloneqq \neg \psi \until \sigma_i$  $i=1,\ldots, k$, such that $\phi = \neg \psi \until \land_{i=1}^k \sigma_i$.  

For each primitive task, the policy for maximizing the probability of satisfying a conditional reachability formula $\varphi_i^q \in \Phi$ can be solved through stochastic shortest path problem in the MDP $M$, referred to as an \ac{ssp} MDP, with a formal definition follows.  \begin{definition} \label{def:ssp}
  A (discounted) \ac{ssp} MDP is defined as a tuple
  $M =(S,A, P, r, \gamma, \goal, \unsafe, s_0)$ where $\goal\subseteq S$ is a set of
  \emph{absorbing goal} states and $\unsafe \subseteq S$ is a set of \emph{absorbing unsafe} states. The transition probability function $P$ satisfies $P(s|s,a)=1$ for all $s\in \goal \cup \unsafe$, for all $a\in A$. The planning problem is to maximize the (discounted) probability of reaching $\goal$ while avoiding $\unsafe$. 
This  objective is equivalent to maximizing the total (discounted) reward with
 the reward function $r: S\times A\rightarrow \reals$ defined as: For
  each $s \in \goal \cup \unsafe$, $r(s,a)=0$ for all $a\in A$, $r(s,a) = \Expect_{s'} \indicator_{\goal}(s')$ for $s\notin \goal$.
  $\gamma \in (0,1]$ is the discounting factor.
\end{definition}

Given $\varphi_i^q = \neg \varphi_{\unsafe(q)}\until \sigma_i $, the corresponding \ac{ssp} MDP shares the same state and action sets with the underlying MDP $M$ that models the system. The transition function is revised from the transition function in the original MDP $M$ by making $\goal =\inside{\sigma_i}$ and $\unsafe  =\inside{\varphi_{\unsafe(q)}}$ absorbing states. Recall that $\inside{\phi}$ is a set of states satisfying the propositional logic formula $\phi$.  
Note when $\gamma\ne 1$, the solution of \ac{ssp} MDP is the solution of a discounted stochastic shortest path problem.  The expected total reward becomes the discounted probability of satisfying the conditional reachability formula.

For stochastic shortest path problems, there exists a deterministic, optimal, Markov policy. However, to compose policies, we use a  class of policies called \emph{entropy-regulated policies}, where softmax Bellman operator is used instead of hardmax Bellman operator. 
Given $\tau$ as the temperature parameter, the optimal value function with softmax Bellman operator satisfies:
 \[
  V^\ast(s) = \tau \log \sum_{a\in A} \exp\{( r(s,a) + \Expect_{s'\sim P(\cdot |s,a)}V^\ast (s'))/\tau\}.
 \]
 The Q-function is:
 \[
 Q^\ast(s,a) = r(s,a) + \Expect_{s'\sim P(\cdot|s,a)} V^\ast(s'), 
 \] 
 and the entropy-regulated optimal policy is:
 \begin{align*}
 \pi^\ast(a|s)  &= \exp \left((Q^\ast (s,a) -V^\ast(s))/\tau\right) \\
 &= \frac{\exp(Q^\ast(s,a)/\tau)}{\sum_{a'} \exp(Q^\ast(s,a')/\tau)}.
 \end{align*}
In the following, by optimal policy/value function, we mean the entropy-regulated optimal policy/value function unless otherwise specified.
 
It is noted that the softmax Bellman operator is also proved  equivalent to the following form \cite{nachum2017bridging}:
 \begin{align*}
     V^\ast(s) = \sum_{a\in A} \pi^\ast(a|s)[r(s,a) - \tau\log\pi^\ast(a|s) \\+ \gamma \Expect_{s'\sim P(\cdot |s,a)}V^\ast (s')],
 \end{align*}
 which means in softmax optimal planning the objective has to trade off maximizing the reward and minimizing the total entropy of the stochastic policy---such a trade off is reflected in the choice of $\tau$. In this case, the construction of reward function needs to be different from the reward defined in Def.~\ref{def:ssp} to reduce the value diminishing problem, which means the entropy of the stochastic policy outweights the total reward for small reward signals in softmax optimal planning. In this paper, we define the reward function for the entropy-regulated MDP as:
\begin{equation}
    \label{eq:entropy_reward}
    r (s,a) =   \alpha\cdot \Expect_{s'} \indicator_{\goal}(s'), \quad \forall s\notin \goal,
\end{equation}
where $\alpha >0$ is a large constant.
 
 For each conditional reachability subtask $\varphi_i^q$, the optimal policy $\pi$ in the corresponding stochastic shortest path MDP is an option $o\coloneqq  \langle I , \pi  , \beta \rangle $, following the definition in \cite{sutton1999between,precup2000temporal} in which $I  \subseteq S$ is an initiation set and is the domain of $\pi$ and $\beta: S\rightarrow [0, 1]$ is the termination
function, defined by $\beta(s)=1$ only if $s\in \unsafe \cup \goal$.  We refer the option for task $\neg \psi \until \sigma$ as $O(\sigma, \psi)$.

\begin{definition}
\label{def:po}
The option $O(\sigma, \psi)$ for subtask $\neg \psi \until \sigma$ is a primitive option if and only if $\sigma$ is an atomic proposition and $\psi$ is the co-safe constraint in the given task DFA.
\end{definition}

\begin{figure}[t] 
%\vspace{0.5cm}
\centering
\resizebox{.38\textwidth}{!}{
\begin{tikzpicture}[->,>=stealth',shorten >=1pt,auto,node distance=2.5cm,
                    semithick]
  \tikzstyle{every state}=[fill=white,draw=black,text=black]

  \node[initial,state] (A) {$q_{init}$};
  \node[state]         (B) [below = 2cm of A] {$q_1$};
  \node[state]         (C) [below right of=B] {$q_3$};
  \node[state]         (D) [below left of=B] {$q_2$};
  \node[state, accepting]         (E) [left=3cm  of D]       {$q_4$};
  \node[state, accepting]         (F) [above right of=C]       {$q_{sink}$};

  \path (A) edge node {$\sigma_1$} (B)
  			edge [dashed] node {$C$} (F)
        (B) edge [loop left] node {$\sigma_1$} (B)
            edge node [left] {$\sigma_3 \backslash \sigma_2$} (C)
            edge [bend right] node [left] {$\sigma_2 \backslash \sigma_3$} (D)
            edge [bend right = 70] node {$\sigma_2\land\sigma_3$} (E)
            edge [dashed] node {$C$} (F)
        (C) edge [loop left] node {$\sigma_1\lor(\sigma_3\backslash\sigma_2)$} (C)
            edge [bend left = 30] node {$\sigma_2$} (E)
            edge [dashed] node {$C$} (F)
        (D) edge [loop left] node {$\sigma_1\lor(\sigma_2\backslash\sigma_3)$} (D)
            edge [bend right] node [left=0.8cm] {$\sigma_3$} (E)
        (D.east) edge [bend right = 110, dashed] node [right =0.9cm] {$C$} (F)
        (E);
        
   \node[draw,text width=8.5cm] at (0,-7) {Primitive task options: $O(\sigma_1, C), O(\sigma_2, C), O(\sigma_3, C)$};
   \node[draw,text width=11cm] at (0,-8) {Non-primitive task options: $O(\sigma_2 \land \sigma_3, C), O(\sigma_2 \backslash \sigma_3, C), O(\sigma_3 \backslash \sigma_2, C)$};

\end{tikzpicture}
}
\caption{The simplified DFA translation of the task $\lnot C\until \left(\Eventually (\sigma_1 \land (\Eventually \sigma_2 \land \Eventually \sigma_3)) \right)$. 
}
\vspace{-4ex}
\label{fig:DFA}
\end{figure}

\begin{example}
Consider the \ac{dfa} in Fig.~\ref{fig:DFA} and the corresponding sc-LTL task specification is $\lnot C\until\left(\Eventually (\sigma_1 \land (\Eventually \sigma_2 \land \Eventually \sigma_3)) \right)$ (reach $\inside{\sigma_1}$ and then reach regions $\inside{\sigma_2}$ and $\inside{\sigma_3}$, always avoid $\inside{C}$). The set of atomic propositions are $\{\sigma_1,\sigma_2,\sigma_3, C\}$. The level sets are $L_0=\{q_4\}$, $L_1 =\{q_1,q_2,q_3\}$, and $L_2 =\{ q_\init\}$. For a given state, for example, $q_1$, the set of labels for progressing transitions are $\{q_1\xrightarrow{\{\sigma_2,\sigma_3\}} q_4\}$. According to Def.~\ref{def:po}, the set of primitive tasks are $\neg C\until \sigma_1, \neg C\until \sigma_2, \neg C\until \sigma_3$. Therefore, three  primitive options are computed as: $O(\sigma_i, C)$, for $i=1,2,3$. In addition, $\neg C\until (\sigma_2\land \sigma_3)$, $\neg C\until (\sigma_2\backslash \sigma_3)$ and $\neg C\until (\sigma_3\backslash \sigma_2)$ are not primitive tasks.

\end{example}

\subsection{Composition of options with  disjunction and conjunction}
\label{subsec: composition}
For a given state $q\in Q$, we have obtained a set of conditional reachability formulas $\varphi_i^q$, for $i=1,\ldots, n$ where $n$ is the total number of primitive tasks generated from $q$. However, a progress transition can be made by satisfying any of the conditional reachability formula. That is to say, we may be interested in synthesizing option that maximizes $\lor_{i=1}^n \varphi_i^q$, or potentially the conjunction/disjunction of a subset of all primitive tasks $\Phi$. A naive approach is to take the new specification $\lor_{i=1}^n \varphi_i^q$, construct a \ac{dfa}, and synthesize the optimal policy using methods \cite{ding2011mdp,baier2008principles} for \ac{mdp} under temporal logic constraints. However, we are interested in  finding a ``good enough'' policy given the new specification via composing existing policies. The problem is formally stated as follows.
\begin{problem}
Given two conditional reachability formulas $\varphi_1\coloneqq  \neg \varphi_{\unsafe} \until \sigma_1$, and $\varphi_2 \coloneqq \neg \varphi_{\unsafe} \until \sigma_2 $, construct a \emph{good enough} policy given the goal of maximizing  the probability of satisfying the disjunction: $\varphi_1\lor \varphi_2 $;
or b) the conjunction: $\varphi_1\land \varphi_2$; or c) the exclusion $\varphi_1\backslash \varphi_2$ or $\varphi_2\backslash \varphi_1.$
\end{problem} 
The definition of ``good enough'' policies will be provided later.  
Here, we consider the case when $\varphi_1$ and $\varphi_2$ share the same set of
unsafe states. Particularly, if $\varphi_1 =  \varphi_i^q$, and $\varphi_2= \varphi_j^q$ for some $i\ne j$ and $q\in Q$, then it is always the case that  $\varphi_1$ and $\varphi_2$ share the same set of
unsafe states. Next, we propose a method for policy composition based on generalized logic conjunction/disjunction \cite{Dujmovi2007}, which is briefly introduced below.

\noindent \paragraph*{Generalized conjunction/disjunction}
\ac{gcd} was introduced in \cite{Dujmovi2007} for quantitative reasoning with
logic formulas. \ac{gcd} is a mapping
$\lambda: [0,1]^n \rightarrow [0,1]$, $n>1$, that has properties
similar to logic conjunction and disjunction. The level
of similarity is adjustable using a parameter $\eta$, called the conjunction
degree (andness). Formally, let $x_1,\ldots, x_n$ be variables
representing the \emph{level of truthfulness} for a set of logic
formulas $\varphi_1,\ldots, \varphi_n$, the \ac{gcd} formula $\lambda(x_1,\ldots, x_n |\eta)$, which
unifies conjunction and disjunction, is defined as,
\begin{align}
\lambda(x_1,\ldots, x_n |\eta) =  \frac{1}{\eta}\log(\sum_{i=1}^n W_i \exp(\eta x_i)),\\ \quad 0<|\eta|<+\infty, \nonumber
\end{align}
where when $\eta\rightarrow \infty$,
$ \lim_{\eta\rightarrow \infty} \lambda(x_1, x_2, ..., x_n|\eta) = x_1\lor x_2\lor
... \lor x_n$ recovers the conventional disjunction, and when
$\eta\rightarrow -\infty$,
$ \lim_{\eta\rightarrow - \infty} \lambda(x_1, x_2, ..., x_n|\eta) = x_1\land
x_2\land ... \land x_n$ recovers the conventional conjunction. For any
$\eta \in (-\infty, +\infty)$, $ \lambda(x_1, x_2, ..., x_n|\eta) $ returns a
\emph{level of truthfulness} of a \ac{gcd}. In addition, parameter $W_i$ is the corresponding weight (or relative importance) of 
the $i$-th formula, for $i=1,\ldots, n$. In this paper, we select $W_i=1$ by assuming that all formulae are equally important.

We use \ac{gcd} to compose a ``good enough'' policy, that is, the optimal policy in semi-MDP planning. 

\begin{definition}\cite{sutton1999between,precup2000temporal} Given an MDP $M = (S, A, P, \gamma, r)$ and a set $O = \{o_i= \{\mathcal{I}_i, \pi_i, \beta_i\},i=1,\ldots, n\}$ of options where $\mathcal{I}_i$ is a set of initial states, $\beta_i: S\rightarrow [0,1]$ is a termination condition and $\pi_i: S\rightarrow \dist{A}$ is a policy in the MDP $M$.  An \emph{option policy} in $M$ is a function $\pi^o: S\rightarrow \dist{O}$. let $\Pi^o$ be the set of option policies in $M$.
Given a reward function $r:S\times A \rightarrow \reals$, an option policy is optimal if and only if it maximizes the total discounted reward:
\begin{equation}
\pi^{o,\ast} = \arg \max_{\pi^o\in \Pi^o} \Expect_{\pi^o} \sum_{t=1}^\infty r(s_t,o_t) 
\end{equation}
where $r(s_t,o_t) = \Expect_{\pi_{o_t}} \left[ r_{t+1}+\gamma r_{t+2} + \ldots \gamma^{k-1}r_{t+k}\right]$ is the total accumulated rewards when the policy $\pi_{o_t}$ of option $o_t$ is applied to the MDP for the duration of $k$ steps.
\end{definition}

\begin{assumption}
\label{assume:absorbing}
Given a conditional reachability subtask $\varphi_i$, the optimal policy $\pi$ that maximizes the probability of satisfying $\varphi_i$ induces in $M$ an absorbing Markov chain $M^{\pi_i}$.
\end{assumption}
In other words, with probability one, the system will visit an absorbing state.
\begin{lemma}
\label{lemma:GCD}
\label{lma:composition}
Assuming \ref{assume:absorbing}, given a set $O = \{o_i= \{\mathcal{I}_i, \pi_i, \beta_i\},i=1,\ldots, n\}$ of options   where  $\mathcal{I}_i=S$, $\pi_i$ is
the softmax optimal policy for maximizing the probability of satisfying $\varphi_i= \neg\varphi_{\unsafe} \until \sigma_i$, \ie, the \ac{ssp} MDP $M_i =(S,A, P, r, \gamma, \inside{\sigma_i}, \inside{\varphi_\unsafe}, s_0)$
$\beta_i(s)= \indicator_X(s) $ where $X = { \bigcup_{j=1}^n \inside{\sigma_j} \cup \unsafe}$ is the
termination function.  In the MDP $M$, the optimal option policy for maximizing the \ac{gcd} $
\lambda(\varphi_1,\ldots,\varphi_n \mid \eta; s)  $  for any $s\in S$, with unit weights, \ie, $W_i=1$, $i=1,\ldots, n$, is % do not need to have $\eta \gg 0$.
 \begin{equation}
 \label{eq:option-policy}
  \pi^o (s)[j] = \frac{\sum_{i=1}^n \exp(\eta Q_i(s, o_j))}{\sum_{o_k\in
      O}\sum_{i=1}^n \exp(\eta Q_i(s, o_k) ) }, \text{ for  }j=1,\ldots, n.
\end{equation}
where $Q_i(s,o_j)$ is the evaluation of policy $\pi_j$ with respect to specification $\varphi_i$.
\end{lemma}
\begin{proof}
For state $s$ and an option $o_j$, by definition of \ac{gcd}, we have \[ \lambda(\varphi_1,\ldots,\varphi_n \mid \eta; s, o_j) = \frac{1}{\eta}\log(\sum_{i=1}^n \exp(\eta \Expect_{ j}[\indicator(s\models\varphi_i)])) .\] 
where the expectation $\Expect_{j} \left[ \cdot \right]$ is taken over paths in the Markov chain $ M^{\pi_j}$. 
 Given the option $o_j$, $\Expect_{j} \indicator(s\models\varphi_i) = Q_i(s, o_j)$---the evaluation of policy $\pi_j$ with respect to specification $\varphi_i$ and thus $\lambda(\varphi_1,\ldots,\varphi_n \mid \eta; s, o_j) = \frac{1}{\eta}\log(\sum_{i=1}^n \exp(\eta Q_i(s,o_j))$. 
 
Next, we distinguish two cases between conjunction and disjunction:

Case I (Disjunction): $\eta > 0$: 
  To maximize $\lambda(\varphi_1,\ldots,\varphi_n \mid \eta; s) $ given an option-only decision rule $\pi^o(s):O\rightarrow [0,1]$, based on the softmax operator, when $\eta >0$, $\pi^o(s)[j] \propto \exp( \lambda(\varphi_1,\ldots,\varphi_n \mid \eta; s, o_j)/\tau)$ where $\tau>0$ is a temperature parameter. When $\tau = 1/\eta$, then 
 \begin{align}
  \pi^o (s)[j] &= \frac{\sum_{i=1}^n \exp(\eta Q_i(s, o_j))}{\sum_{o_k\in
      O}\sum_{i=1}^n \exp(\eta Q_i(s, o_k) ) }, \quad \\ &\text{for  }j=1,\ldots, n. \nonumber
  \end{align}
  which is the same as in \eqref{eq:option-policy}.

  Case II(Conjunction): $\eta < 0$: In this case, we have 
  \begin{multline} 
\lambda(\varphi_1,\ldots,\varphi_n \mid \eta; s, o_j)  \\
= -  \frac{1}{\abs{\eta}}\log(\sum_{i=1}^n \exp(- \abs{\eta} \Expect_{ j}[\indicator(s\models\varphi_i)])) .\end{multline} To maximize $ \lambda(\varphi_1,\ldots,\varphi_n \mid \eta; s) $ is equivalent to minimize $\frac{1}{\abs{\eta}}\log(\sum_{i=1}^n \exp(- \abs{\eta} x_i)) $ where $x_i$ is the level of truthfulness for formula $\varphi_i$.  Further, minimizing $\frac{1}{\abs{\eta}}\log(\sum_{i=1}^n \exp(- \abs{\eta} x_i)) $ is equivalent to minimizing $\log(\sum_{i=1}^n \exp(\eta x_i))$ as $\eta=-\abs{\eta}$, which is exactly the opposite case to that of disjunction. Thus, the optimal option policy satisfies
  $\pi^o(s)[j] \propto \exp( -  \lambda(\varphi_1,\ldots,\varphi_n \mid \eta; s, o_j)/\tau)$ (softmin operator) where $  \tau>0$  is a temperature parameter. When $\tau = - 1/\eta$, then given $\lambda(\varphi_1,\ldots,\varphi_n \mid \eta; s, o_j) = Q_i(s, o_j)$, for $j=1,\ldots, n$,
   \begin{align*}
  \pi^o (s)[j] &= \frac{\sum_{i=1}^n \exp( -  Q_i(s, o_j)/\tau)}{\sum_{o_k\in
      O}\sum_{i=1}^n \exp(-Q_i(s, o_k)/\tau ) }\\
      & = \frac{\sum_{i=1}^n \exp(  \eta  Q_i(s, o_j))}{\sum_{o_k\in
      O}\sum_{i=1}^n \exp( \eta Q_i(s, o_k) ) },    
  \end{align*}
  which is the same as in \eqref{eq:option-policy}.  Thus the proof is completed.
\end{proof}

Intuitively, for the case of disjunction,  this policy makes sense because given  $\pi_j$ is the optimal policy for satisfying $\varphi_j$, $\pi^o(s)$ selects policy $j$ with a likelihood proportional to  $\exp (\eta Q_j(s, o_j))$ plus some bonus $\sum_{i\ne j}\exp (\eta Q_j(s, o_i))$ obtained by satisfying other specifications. Given two specifications $\varphi_1$ and $\varphi_2$, since the disjunction can be satisfied by satisfying only one of these two, then this policy exponentially prefers $\pi_1$ to $\pi_2$ if $\varphi_1$ has a higher probability to be satisfied.

The situation is complicated for conjunction. The conjunction of two formulas, $\varphi_1\coloneqq \lnot \varphi_\unsafe \until \sigma_1$ and $\varphi_2\coloneqq \lnot \varphi_\unsafe \until \sigma_2$, is $(\lnot \varphi_\unsafe \until \sigma_1)\land (\lnot \varphi_\unsafe \until \sigma_2)$. For any state, the planner will select the option $i$ with a probability that is inverse proportional to $Q(s,o_i)$, \ie, if $\varphi_1$ has a lower probability to be satisfied, then option $1$ has a higher probability to be chosen.   Once it reaches $\inside{\sigma_1}$, it will select option $2$ with a higher probability because $Q(s,o_2) < Q(s,o_1) =1$ for $s\in \inside{\sigma_1}$ to force a visit to $\inside{\sigma_2}$. 
However, without memory, the planner will alternate between two options indefinitely, or until it reaches an unsafe region. Thus the conjunction on multiple memoryless options  will require additional memory to manage the switching condition of terminating function among goals. However,  when the intersection $\inside{\sigma_1\land \sigma_2}\ne \emptyset$ and either option has a nonzero probability of reaching the intersection, a memoryless composed option may eventually reach a state in $\inside{\sigma_1\land\sigma_2}$. Thus, we may approximate the solution of $\lnot \varphi_\unsafe \until (\sigma_1
\land \sigma_2)$ with a memoryless composed option for the conjunction.  In this paper, we only focus on the memoryless option, further discussions on the additional memory method will be included in the future work. 

Based on the proof of Lemma~\ref{lma:composition}, we can further show that the GCD method is indeed invertible to compute the exclusion $\varphi_1\backslash\varphi_2$. Since $(\varphi_1\land \varphi_2) \lor(\varphi_1\backslash \varphi_2) = \varphi_1$,  we can apply generalized disjunction in Eq. (2) to the MDP $M$, then
\begin{align*}
\centering
\exp(\eta \Expect_{ 1}[\indicator(s\models\varphi_1)]) &\approx \exp(\eta \lambda(\varphi_1\land \varphi_2, \varphi_1\backslash \varphi_2 |\eta;s,o_1) \\ 
    &=  \exp(\eta \Expect_{ 1}[\indicator(s\models\varphi_1\land\varphi_2)]) \\&+ \exp(\eta \Expect_{ 1}[\indicator(s\models\varphi_1\backslash\varphi_2)]).
\end{align*}
Therefore, the Q function of task exclusion can be computed by
\begin{align*}
  Q(s, o_1) &= \Expect_{ 1}[\indicator(s\models\varphi_1\backslash\varphi_2)]\\&\approx
\frac{1}{|\eta|}\log(\exp(\eta \Expect_{ 1}[\indicator(s\models\varphi_1)]) 
\\&- \exp(\eta \Expect_{ 1}[\indicator(s\models\varphi_1\land\varphi_2)])).
\end{align*} 

Finally, the set of options $\cal O$ includes both primitive options---one for each primitive tasks and composed options using \ac{gcd}. The set of actions $A$ is now augmented with options $\cal O$, and the optimal policy can be obtained by solving the following planning problem in the product MDP with an augmented action space. 

\begin{definition}%[Product \ac{mdp} with actions and options]
\label{def:mixed}
Given a labeled MDP $M = (S,A, P,\mu_0, L)$ and an sc-LTL formula $\varphi$, represented by a \ac{dfa} $\calA_\varphi= (Q,\Sigma, \delta,q_0, F)$,  a set $\calO$ of options,  the \emph{product MDP with macro- and micro-actions} is defined by 
\[
M\ltimes \calA_\varphi = (S\times Q, A \cup \calO, \bar P, \bar \mu_0 ),
\]
where the probabilistic transition function is defined by: 
$
\bar P((s',q')|(s,q), a) = P(s'|s,a) \text{ if } q' =\delta(q, L(s')), \;  a\in A,
$  and 
$
\bar P((s',q')|(s,q), o) = \prod_{t=0}^k P (s_{t+1}|s_t, a_t) $ where $ s_0=s$ and $s_{k+1}=s'$,  $k$ is the number of steps that are taken under the (policy) of option $o$ before being interrupted by triggering a discrete transition in the automata,  $\{s_t,a_t, t=0,\ldots \}$ is  Markov chain induced by the policy of option $o$, and $q' = \delta(q,L(s'))$.
 The reward function is defined by $r((s,q),a) = \Expect_{(s',q')}\indicator(q'\in F)$, when $q\notin F$, and $r((s,q),o) =\Expect_{\pi_o } \left[\sum_{t=0}^{k-1} \gamma^t  r(s_t,a_t) \right]$ is the total accumulated rewards when the policy $\pi_o$ of option $o$ is applied to the MDP for the duration of $k$ steps before it is interrupted.
\end{definition}
Note that when the chain is absorbing for any policies of options, then the discounting factor $\gamma$ can be set to $1$. By setting $\gamma\ne 1$, we will encourage the behavior of satisfying the specification in a less number of steps.

\begin{remark}
The planning is performed in the product MDP with both actions and options. It is ensured to recover the optimal policy had only actions being used. Having options helps to speed up the convergence. Note that even if self-loops in the \ac{dfa} have not been considered in generating primitive and composed options, the optimality of the planner will not be affected.
\end{remark}

\begin{remark}
 For the labeled MDP in Def.~\ref{def:mixed}, let $N_S, N_Q, N_A, N_{\calO}$ denote the size of $S, Q, A, \calO$. In value iteration algorithm, the space complexity of the planning performed with actions only is $\mathcal{O}(N_S N_Q N_A)$, while the planning performed with both actions and options consumes $\mathcal{O}(N_S N_Q (N_A+N_O) + N_S N_A N_O)$ of memory,  for storing the policy with both action and options and these option policies. If we assume $\mathcal{O}(N_Q)= \mathcal{O}(N_A)=\mathcal{O}( N_{\calO})= \mathcal{O}(\Bar{N})$, and $\mathcal{O}(N_S) \gg \mathcal{O}(\Bar{N})$, the two methods will have the same space complexity, that is $\mathcal{O}(\Bar{N}^2 N_S)$.
% \begin{align*}
%     &\mathcal{O}(N_S N_Q N_A) = \mathcal{O}(N_S N_Q (N_A+N_O) + N_S N_A N_O) \\
%     &= \mathcal{O}(\Bar{N}^2 N_S).
% \end{align*}
% comment: a bit skeptical whether we can write equality using big O notation. 
\end{remark}

\section{Case Study}

\begin{figure}[t]
%\vspace{0.5cm}
\includegraphics[scale=0.23]{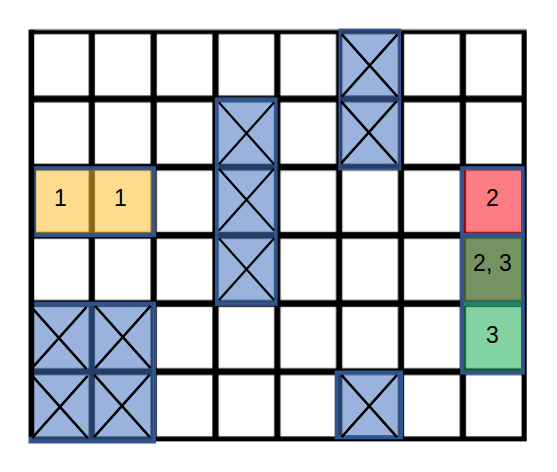}
\centering
\caption{The state space setting for the task cases in a 2-D grid world. 
}
\label{fig:gridworld}\vspace{-2ex}
\end{figure}

\begin{figure}[t]
%\vspace{0.25cm}
\hspace{0.1cm}
\includegraphics[scale=0.3]{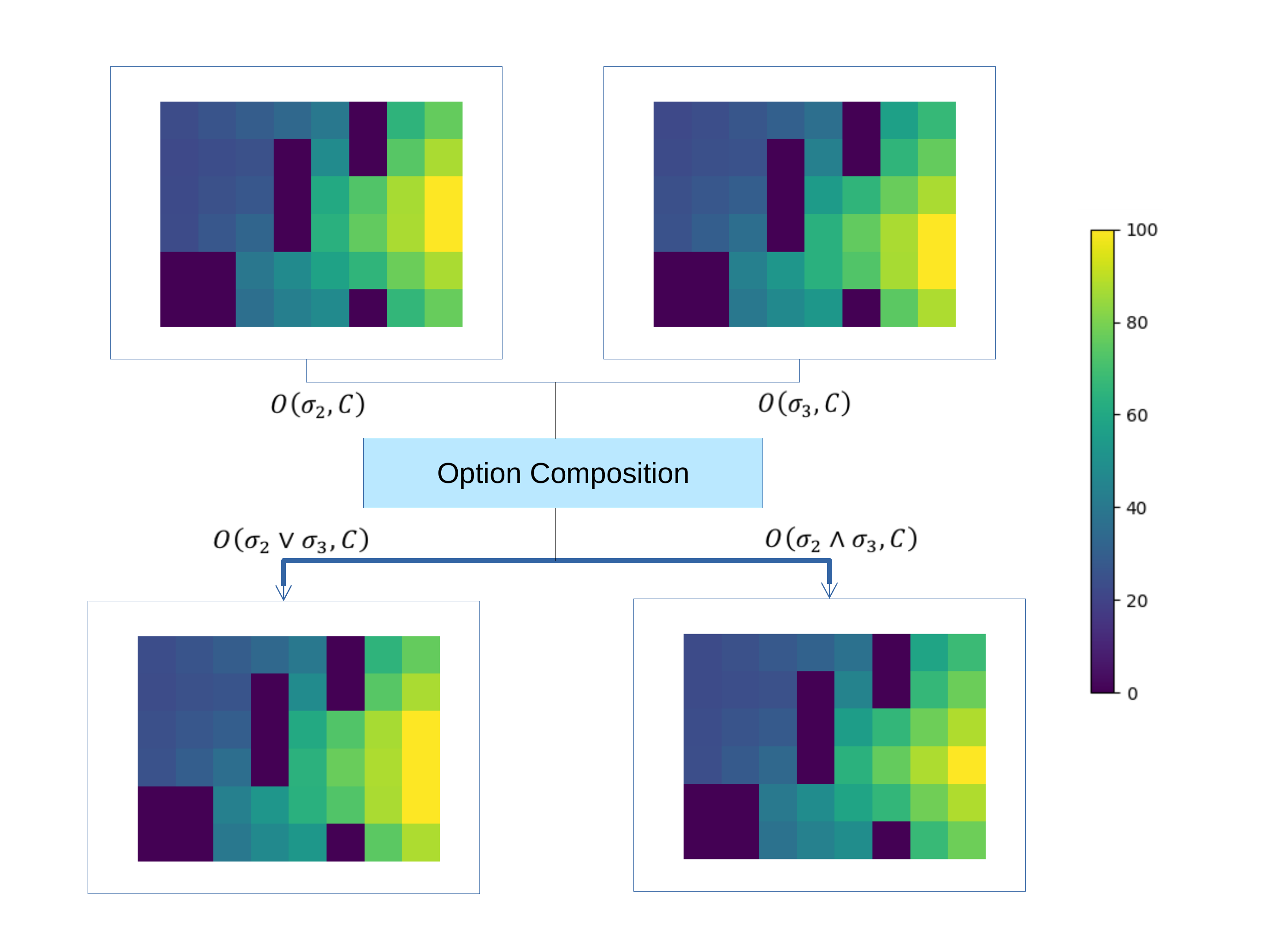}
\centering
\caption{The figure shows the converging distribution of entropy regularized value functions of primitive options $O(\sigma_2, C)$, $O(\sigma_3, C)$ and their compositions to approximate general conjunction and disjunction. }
\label{fig:result1}
\vspace{-4ex}
\end{figure}

This section illustrates our compositional planning method using robotic motion planning problems. All experiments in this section are performed on a computer equipped with an Intel$^{\tiny{\textregistered}}$ Core$^{\text{TM}}$ i7-5820K and 32GB of RAM running a python 3.6 script on a 64-bit Ubuntu$^{\tiny{\textregistered}}$ 16.04 LTS. 

 The environment is modeled as a 2D grid world, shown in Fig.~\ref{fig:gridworld}. The robot has actions: Up, Down, Left, and Right. With probability 0.7, the robot arrives at the cell it intended with the action and has a probability of $(1-0.7)/(|Adj|-1)$ to transit to other adjacent locations, where $|Adj|$ is the number of adjacent cells, including the current cell when the action is applied.  Especially, when the transition hits the boundary of the grid world, the probability under that transition adds to the probability of staying put. The discounting factor $\gamma$ is fixed $0.9$. Fig. \ref{fig:gridworld} shows a grid world of  size $6\times 8$. In this grid world, there are a set of obstacles (unsafe grids) marked with cross signs and several regions of interests, marked with numbers. The cell marked with number $i$ is labeled with symbol $\sigma_i$, for $i=1,2,3$.
 Region marked with number $2,3$ is  the nonempty intersection satisfying $\sigma_2 \land \sigma_3$. 

We consider three sc-LTL tasks $\varphi = \{\varphi_1, \varphi_2, \varphi_3\}$ where
\[\varphi_1 \coloneqq \lnot C\until\left(
 \Eventually ( \sigma_1 \land (\Eventually \sigma_2 \land \Eventually \sigma_3) )\right),
\]
(reach $\inside{\sigma_1}$ and then reach regions $\inside{\sigma_2}$ and $\inside{\sigma_3}$, while avoiding $\inside{C}$.)
\[\varphi_2 \coloneqq \lnot C\until\left(
 \Eventually( (\sigma_1 \lor \sigma_3) \land \Eventually \sigma_2 )\right),
\]
(reach either $\inside{\sigma_1}$ or $\inside{\sigma_3}$ and  then reach $\inside{\sigma_2}$, while avoiding  $\inside{C}$.)
\[\varphi_3 \coloneqq  \lnot C\until\left(
\Eventually ((\sigma_1 \lor \sigma_2) \land \Eventually (\sigma_2 \land \sigma_3)) \right),
\]
(reach  either $\inside{\sigma_1}$ or $\inside{\sigma_3}$  and   then reach  either $\inside{\sigma_2}$ or $\inside{\sigma_3}$, while avoiding  $\inside{C}$.)

Figure \ref{fig:DFA}.a shows the \ac{dfa} of $\varphi_1$. We omit the \ac{dfa}s for $\varphi_2$ and $\varphi_3$ given the limited space.  Based on the set of tasks, the following set of primitive tasks are generated:
% \todo[inline]{HERE LIST ALL THE OPTIONS YOU GENERATED WITHOUT COMPOSITION}  
\begin{align*}
    \phi_1\coloneqq \lnot C \until \sigma_1, \quad \phi_2\coloneqq \lnot C \until \sigma_2, \quad \phi_3\coloneqq  \lnot C \until\sigma_3. \quad
\end{align*}
For each conditional reachability specifications, we formulate the \ac{ssp} MDP and compute the softmax optimal policy using temperature parameter $\tau=1$ and the reward function defined in Eq. \ref{eq:entropy_reward}, where parameter $\alpha$ is selected to be $100$ to prevent the reward being outweighted by the entropy term.  

Our first experiment is to demonstrate the composition of policies based on GCD.

\paragraph*{Policy composition} We use composition of $ \phi_2=  \lnot C  \until \sigma_2 $, $ \phi_3=  \lnot C \until\sigma_3$ to generate the options $\lnot C\until (\sigma_2\oplus \sigma_3)$ where $\oplus \in \{\land, \lor\}$. To validate that the composed policies are ``good enough'', we compare  the values of the optimal policies for these two formulas, computed using standard value iteration, and the values of  the  composed policies obtained via Lemma~\ref{lma:composition}. For comparison, we consider relative errors $e_{2,\oplus} = \norm{V^{\pi^\oplus}- V^{\pi^{o,\oplus}}}_2/\norm{V^{\pi^{o,\oplus}}}_2$ and $ e_{\infty,\oplus} = \norm{V^{\pi^{\oplus}}- V^{\pi^{o,\oplus}}}_\infty/\norm{V^{\pi^{o,\oplus}}}_\infty$. 
 We have  $e_{2,\lor }\approx 10^{-4}$, $e_{\infty, \lor}\approx 10^{-4}$, $e_ {2,\land} \approx 10^{-3}$, $e_{\infty,\land} \approx 10^{-3}$. 

Figure \ref{fig:result1} shows heat maps comparing two option value functions for the case of disjunction or conjunction.  The shaded areas represent globally unsafe regions with $V$ values always fixed zero during the iteration. In Fig. \ref{fig:result1}, all value distributions are in the range between $0$ to $100$ because we scaled the reward of $1$ by $100$ to avoid entropy term outweighing the total reward. From Fig.~\ref{fig:result1} (c), it is shown that the value of regions marked by either $2$ or $3$ is highest, corresponding to the disjunction. In the case of conjunction in Fig.~\ref{fig:result1} (d), the intersection of regions marked by $2$ and $3$ has the highest value.

Next, we compare the convergence between three different planning methods for three task specifications: planning with only micro-action (action), planning with macro-actions (primitive and composed options), and planning with both micro- and macro-actions (mixed). In addition, we compared the optimality of these planners with optimal planning with only micro-action using hardmax Bellman operator as the baseline. The results to be compared are the speed of convergence and the optimality of the converged policy. 

\begin{figure}[t]
%\vspace{0.2cm}
\includegraphics[scale=0.45]{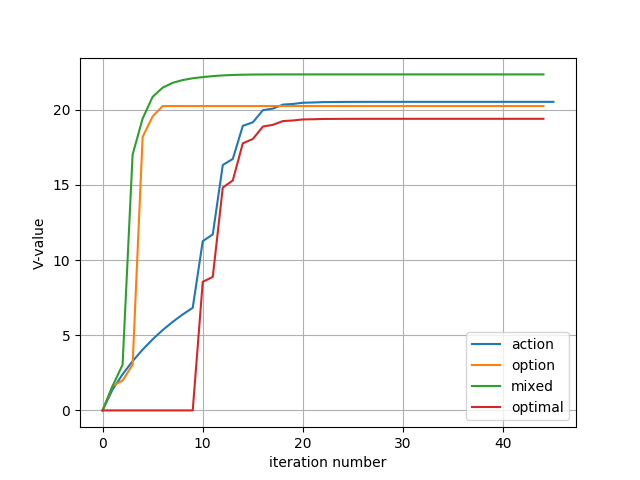}
\centering
\caption{Value iteration of four different planners for task $\varphi_1$.
}
\vspace{-3ex}
\label{fig:result2}
\end{figure}

Figure~\ref{fig:result2} shows the convergence of value function evaluated at the initial state with $s_0 = (3, 3)$ given specification $\varphi_1$.  It shows that among all three methods, the mixed planner converges the fastest. Both option and mixed planners converge much faster than action planner: The action planner converges after about 20  iterations, while the other two converges after 6-9  iterations. It is also interesting to notice that the policy obtained by the mixed planner achieves a higher value comparing to the action planner. This is because the entropy of policy weighs less in the policy obtained by the mixed planner comparing to that obtained by the action planner in softmax optimal planning. Moreover, the influence of entropy can also be observed between softmax action planner (action) and hardmax action planner (optimal), where the two planners converge almost at the same rate but to different values since softmax adds additional policy entropy to the total value.

Table \ref{tab:result} compares the performance of three planners in the given three  tasks. The entry $p$ refers to the probability of satisfying the specification from an initial state $s_0$ under the optimal policies obtained by three planners. Number $n$ is the number of value iterations taken for each method to converge with a pre-defined error tolerance threshold $0.001$.  Number $t$ shows the CPU time costs. 

In converging iteration numbers and CPU times, the advantage of option planner outperforms the other two significantly. Considering the additional time cost from learning the primitive options, the experiment shows that every single option takes in average 30 iterations to converge in a $6\times8$ option state space, and in total costs $0.5$ seconds to compute all the  options for primitive tasks. However, these options only need to be solved for once and are reused across three tasks. The composition of options takes negligible computation time ($0.001$ seconds on average for each composition). The performance loss of option planner, comparing with the global optimal planner (using hardmax Bellman), is only $13\%$ of the optimal value for task $\varphi_1$ and negligible for tasks $\varphi_2$ and $\varphi_3$. %We also observe that for the three task with shared primitive options, all primitive options took $0.21$ seconds to compute. 
Composition makes temporal logic planning flexible: If we change a task from $\varphi_1$ to $\varphi_2$, then the option and mixed planner can quickly generate new, optimal policies without reconstructing primitive option.

\begin{table}[t]
\vspace{0.2cm}
\caption{Comparison: Performance, convergence rate, and computation cost}
    \centering
    \resizebox{.49\textwidth}{!}{
    \begin{tabular}{|*{11}{c|}} 
        \hline
         &
         \multicolumn{3}{|c}{$p(\varphi|s_0)$} & \multicolumn{3}{|c|}{iteration($n$)}&
         \multicolumn{3}{|c|}{runtime($t:sec$)} \\
        \hline
        $\varphi$ & $\varphi_1$ & $\varphi_2$ & 
        $\varphi_3$ & $\varphi_1$ & $\varphi_2$ & 
        $\varphi_3$ &
        $\varphi_1$ & $\varphi_2$ & 
        $\varphi_3$\\ 
        \hline
        $Optimal$ &0.54 &0.88 &0.89 &46 &24 &24 &0.6 &0.28 &0.22\\ 
        \hline
        $Action$ &0.35 &0.87 &0.87 &45 &25 &23 &0.6 &0.32 &0.21 \\
        \hline
        $Option$ &0.47 &0.88 &0.88 &8 &6 &5 &0.1 &0.03 &0.02\\
        \hline
        $Mixed$ &0.41 &0.88 &0.87 &34 &18 &15 &0.4 &0.17 &0.12\\
        \hline
    \end{tabular}
    }
    \label{tab:result}
\end{table}

\begin{table}[t]
\vspace{0.15cm}
\caption{The error (2-Norm and $\infty$-Norm) between policies obtained with different methods}
    \centering
    \resizebox{.49\textwidth}{!}{
    \begin{tabular}{|*{11}{c|}} 
        \hline
         &
         \multicolumn{3}{|c|}{Option $\pi_1$ vs Action $\pi_2$}&
         \multicolumn{3}{|c|}{Mixed $\pi_1$  vs Action $\pi_2$} \\
        \hline
        $\varphi$ & $\varphi_1$  & $\varphi_2$ & 
        $\varphi_3$ & $\varphi_1$ & $\varphi_2$ & 
        $\varphi_3$\\ 
        \hline
        $2$-Norm &0.0076 &0.0345 &0.0394 &0.0022 &0.0009 &0.001\\ 
        \hline
        $\infty$-Norm &0.0249 &0.1755 &0.17 &0.0088 &0.0071 &0.0067\\
        \hline
    \end{tabular}
    }
    \label{tab:result2}
\end{table}
 
Last, we evalute and compare policies generated by option, mixed and action planners. Table \ref{tab:result2} shows the relative errors in 2-norm and infinite-norm, \ie, $e(\pi_1,\pi_2) =  \norm{V^{\pi_1} - V^{\pi_{2}}}/\norm{V^{\pi_2}}$.
Both the option planner and mixed planner have negligible deviation to (less than $3\%$) to the action planner, while the option planner is clearly less similar to the action planner comparing to the mixed planner, especially on the $\infty$-norm error. 

\section{CONCLUSIONS}
This paper presents a compositional method for MDP planning constrained by sc-LTL specifications. The method  formally relates the composition of stochastic policies for logical task specifications and  generalized conjunction/disjunction (GCD) in logic. We show that the composition based on GCD is equivalent to the semi-MDP planning under the softmax Bellman operator. 
The semi-MDP planning with both primitive options and composed options achieves much faster convergence, comparing to planning with actions, or a mixture of actions and options, with a relatively small performance loss. Besides, our compositional planning method brings in more flexibility in re-using and composing options for one task in a different task in the same stochastic system with the same labeling function. 

 Although composed options may not be optimal, the convergence to the global optimal policy is guaranteed as the options framework uses both macro-actions/options and micro-actions (actions in the original MDP). The future direction along this line of work is to exploit the compositional planning in model-free reinforcement learning and to improve the scalability of the planning method by replacing value/policy iteration with approximate dynamic programming \cite{bertsekas2008neuro}. We will further investigate finite-memory policy composition   to handle the issue raised from conjunction-based composition.

\addtolength{\textheight}{-12cm}   % This command serves to balance the column lengths

%%%%%%%%%%%%%%%%%%%%%%%%%%%%%%%%%%%%%%%%%%%%%%%%%%%%%%%%%%%%%%%%%%%%%%%%%%%%%%%%

%%%%%%%%%%%%%%%%%%%%%%%%%%%%%%%%%%%%%%%%%%%%%%%%%%%%%%%%%%%%%%%%%%%%%%%%%%%%%%%%


% \section*{ACKNOWLEDGMENT}

%%%%%%%%%%%%%%%%%%%%%%%%%%%%%%%%%%%%%%%%%%%%%%%%%%%%%%%%%%%%%%%%%%%%%%%%%%%%%%%%

\bibliographystyle{plain} %IEEEtran
\bibliography{refs}

% \begin{thebibliography}{99}

% \end{thebibliography}

\end{document}